\documentclass[12pt]{article}
\usepackage{amsfonts, amsmath, amsthm, amssymb, relsize, scalerel, graphics, tabu}
\usepackage[table]{xcolor}
\usepackage [english]{babel}
\usepackage{cjhebrew}
\usepackage{mathrsfs}
\usepackage{enumerate}
\usepackage [autostyle, english = american]{csquotes}
\usepackage{harpoon}

\newtheorem{theorem}{Theorem}
\newtheorem{proposition}[theorem]{Proposition}
\newtheorem{corollary}[theorem]{Corollary}
\newtheorem{lemma}[theorem]{Lemma}

\theoremstyle{definition}
\newtheorem{example}[theorem]{Example}

\newtheorem*{remark}{Remark}

\title{A Functional Decomposition of Finite Bandwidth Reproducing Kernel Hilbert Spaces}
\author{
Gregory T. Adams\\
Department of Mathematics\\
Bucknell University\\
{\tt adams@bucknell.edu}\vspace{3mm}\\
Nathan A. Wagner\\
Department of Mathematics and Statistics\\
Washington University in St Louis\\
{\tt nathanawagner@wustl.edu}
}

\begin{document}

\maketitle

\begin{abstract}
In this work, we consider ``finite bandwidth" reproducing kernel Hilbert spaces which have orthonormal bases of the form 
$f_n(z)=z^n \prod_{j=1}^J \left( 1 - a_{n}w_j z \right)$, where $w_1 ,w_2, \ldots w_J $ are distinct points on the circle $\mathbb{T}$ and  
$\{ a_n \}$ is a sequence of complex numbers with limit $1$.   We provide general conditions based on a matrix recursion that guarantee such spaces contain a functional multiple of the Hardy space. Then we apply this general method to obtain strong results for finite bandwidth spaces when 
$\lim_{n\rightarrow \infty} n (1-a_n)=p$. In particular, we show that point evaluation can be extended boundedly to precisely $J$ additional points on $\mathbb{T}$ and we obtain an explicit functional decomposition of these spaces for  $p>1/2$ in analogy with a previous result in the tridiagonal case due to Adams and McGuire.  We also prove that multiplication by $z$ is a bounded operator on these spaces and that they contain the polynomials. 
\end{abstract}

\vspace{10pt}

\section{The Problem}\label{theproblem}

If $K(z,w)$ is a function defined on an open disc about the origin which is analytic in $z$ and coanalytic in $w$, then $K$ has a power series representation 
$K(z,w)=\sum_{j=0}^{\infty}\sum_{k=0}^{\infty}a_{j,k}z^j\overline{w}^k.$ In the case that $A=(a_{j,k})$ is a bounded matrix, it is an easy exercise to check that $A$ is positive semi-definite on $\ell^2$
if and only if the function $K$ is, and in this case by the Moore-Aronszajn Theorem the function $K$ is the kernel for a reproducing kernel Hilbert space $H(K)$ (see \cite{nA50}). In this case, the space $H(k)$ consists of analytic functions on a domain containing  a disk about the origin in $\mathbb{C}$. Recall the well-known fact that if $\{f_n\}$ is an orthonormal basis for the reproducing kernel Hilbert space ( RKHS) of functions $H(K)$ associated with $K$, then $K(z,w)= \sum_{n=0}^{\infty}f_n(z) \overline{f_n(w)}$ \cite{Paulsen}. 
Conversely, if $A$ can be factored as  $A = LL^*$ where $L$ has no kernel, then the columns of $L$ give the Taylor coefficients of an orthonormal basis for $H(K)$ \cite{Adams-McGuire-Paulsen}.  In fact, $H(K)$ can be identified with the range space of $L$ in a very natural way \cite{Adams-McGuire-Paulsen}.  This range space identification will lie at the heart of most of our computations.
\par
The Cholesky algorithm always allows for a factorization of a positive definite matrix $A = LL^*$  with $L$ lower triangular.
If $A$ has finite bandwidth $2J+1$, then $L$ is lower triangular with $J+1$ non-trivial diagonals and we speak of a ``bandwidth-$2J+1$" kernel  $K$.  In particular, we say an analytic kernel $K$  is of finite bandwidth $2J+1$ if there exists an orthonormal basis of polynomials for $H(K)$ of the form $$\{f_n(z)=(b_{0,n}+b_{1,n}z+..+b_{J,n}z^J)z^n\}.$$ The simplest case where the space $H(K)$ has bandwidth $1$ was extensively studied by Shields in \cite{aS74} in the context of multiplication operators. Such spaces are referred to as \textit{diagonal spaces} and have orthonormal bases consisting of monomials.

\par
In the context of bandwidth-$2J+1$ analytic kernels, the  \textit{natural domain} of $H(K)$ is given by $\text{Dom}(K)=\{z \in \mathbb{C}: \sum_{n=0}^{\infty}|f_n(z)|^2<\infty\}.$  Adams and McGuire established that the natural domain for $H(K)$ is a disk about the origin with up to $J$ additional points \cite{Adams-McGuire2}.  They explored the $J=1$ case and gave an interesting family of kernels $K$ where $H(K)$ is a nontrivial extension of a diagonal space \cite{Adams-McGuire1}. In this paper, we show how to generalize their results to higher bandwidths.
\par
Now we can state the problem of interest. Throughout this work,  $z_1 ,z_2, \ldots,  z_J $ will be distinct points on the unit circle $\mathbb{T}$ and  $w_1 ,w_2, \ldots,  w_J $ will be the corresponding conjugates. The sequence of complex numbers  $a_0, a_1, \ldots $ will be a  sequence converging to $1$ so that $1-a_j$ is nonvanishing.
Define  
$$\phi(z)= \prod_{j=1}^J \left( 1 - w_j z \right)=\sum_{k=0}^J\beta_k z^k,$$
and 
$f_n(z) = z^n\phi(a_n z)$. We will follow the notational convention that $\beta_j=0$ if $j<0$ or $j>J$.
Then
$$K(z,w)=\sum_{n=0}^{\infty}f_n(z) \overline{f_n(w)}$$   is a bandwidth-$2J+1$ kernel for a RKHS $H(K)$ with orthonormal basis
$\{f_0,f_1, \ldots \}.$
\par
Theorems \ref{containment} and \ref{decomposition} show that in the case where
\newline
$\lim_{n \rightarrow \infty}n(1-a_n)=p$ and $p>1/2$,  $H(K)$ has natural domain 
$\mathcal{D}=\mathbb{D} \cup \{z_1,z_2, \ldots z_J \}$ and decomposes as
$$H(K)= \phi(z) H^2(\mathbb{D})+\mathbb{C} K(z,{z}_1)+\mathbb{C} K(z,{z}_2) + \cdots  +\mathbb{C} K(z,{z}_J).$$
Moreover, in this case, multiplication by $z$ is a bounded operator and the polynomials are contained in $H(K)$.
 \par
These results generalize those in \cite{Adams-McGuire1} and \cite{S} to higher bandwidth and more general weight sequences. This leads to a very nice functional characterization of certain finite bandwidth spaces. The primary innovation in this work is the use of matrix recursion to bound the norm of infinite dimensional matrices, a program which was started in \cite{S}.
Key also is the role played by the combinatorial Theorems \ref{Louck} and \ref{homoSum}.

\section{Preliminaries}
\par
The first result shows that the restrictions of the functions in $H(K)$ to the disc $\mathbb{D}$ are in the Hardy space.
\begin{proposition} 
$H(K) \subset H^2(\mathbb{D})$. 
\end{proposition}
\begin{proof} If $f \in H(K)$, then there exists an $\ell^{2}$ sequence $\{\alpha_n \}$ such that $f=\sum_{n=0}^{\infty} \alpha_n f_n$. Thus:
\begin{small}
\begin{eqnarray*}
 f (z)& = & \sum_{n=0}^\infty \alpha_n f_n(z)\\
& = & \sum_{n=0}^\infty \alpha_n \left( \sum_{k=0}^J  \beta_k a_n^k z^{n+k}\right)\\
&= &  \sum_{n=0}^\infty \left( \sum_{k=0}^J \alpha_{n-k} \beta_k a_{n-k}^k \right) z^n \\
&= &  \sum_{n=0}^\infty  \widehat{\alpha}_n z^n .
\end{eqnarray*}
\end{small}

By the Cauchy-Schwarz inequality,  $| \widehat{\alpha}_n |^2 \le c^2 \sum_{k=0}^J | \alpha_{n-k} |^2$, where $c$ is a constant that depends only on the zeros $z_1 ,z_2, \ldots,  z_J $ and the sequence $\{a_n\}$. Thus,  $\sum_{n=0}^\infty | \widehat{\alpha}_n |^2 \le (J+1) c^2 \sum_{n=0}^\infty | \alpha_{n} |^2$ and $f$ is in $H^2(\mathbb{D})$.
\end{proof}

\par
Given the basis $f_n(z)=\phi(a_n z)z^n$ and the fact that $a_n \rightarrow 1$ it is reasonable to ask when functions of the form $\phi(z)f(z)$ for $f \in H^2(\mathbb{D})$ are in $H(K)$.  The rate of convergence of $a_n$ to $1$ is crucial in assessing when this is the case.  Douglas' Range Inclusion Lemma (see \cite{Douglas}) will provide the major tool to answer this question.
\par
To this end, let $L$ be the matrix whose $n$th column consists of the Taylor coefficients of $f_n(z)$ and let $\widehat{L}$  be the matrix whose $n$th column consists of the Taylor 
coefficients of $z^n\phi(z)$.  By Douglas' Lemma, $\phi(z)H^2(\mathbb{D}) \subset H(K)$ if and only if there is a bounded matrix $C=\left(c_{j,k}\right)_{j,k \ge 0}$ such 
that $\widehat{L}=LC$.  Solving this equation for $C$ is complicated and will involve a recursion.  First note that $L$ and $\widehat{L}$ are both lower triangular which implies that $C$ is as well.  So one must solve
$$\begin{pmatrix} 
\beta_{0} & 0 & 0 &  \cdots \\
\beta_{1} & \beta_{0}  & 0  &  \cdots \\
\beta_{2} &\beta_{1} & \beta_{0}  &  \ddots \\
\vdots &\vdots &\vdots &\ddots&  \ddots\\
\beta_{J}  & \beta_{J-1}  &\beta_{J-2}   &  \ddots\\
0  & \beta_{J}  &\beta_{J-1}  &  \ddots\\
0  & 0 &\beta_{J}  &  \ddots\\
\vdots & \vdots & \ddots &   \ddots
 \end{pmatrix} =
\begin{pmatrix} 
\beta_0 & 0 & 0 &  \cdots \\
\beta_1a_0 & \beta_0  & 0  &  \cdots \\
\beta_2a_0^2 &\beta_1a_1 & \beta_0  &  \ddots \\
\vdots &\vdots &\vdots &\ddots&  \ddots\\
\beta_Ja_0^J  & \beta_{J-1}a_1^{J-1}  &\beta_{J-2}a_2^{J-2}   &  \ddots\\
0  & \beta_Ja_1^J  &\beta_{J-1}a_2^{J-1}  &  \ddots\\
0  & 0 &\beta_Ja_2^J  &  \ddots\\
\vdots & \vdots & \ddots &   \ddots
 \end{pmatrix}   \begin{pmatrix} 
c_{0,0} & 0 & 0 &\cdots \\
c_{1,0} & c_{1,1} & 0& \ddots\\
c_{2,0} & c_{2,1} & c_{2,2} & \ddots\\
c_{3,0} & c_{3,1} & c_{3,2} & \ddots\\
c_{4,0} & c_{4,1} & c_{4,2} & \ddots\\
c_{5,0} & c_{5,1} & c_{5,2} &  \ddots\\
\vdots & \vdots & \vdots & \vdots\\
\end{pmatrix}. $$
 for $C$.

Considering the $n$th column of matrix $C$ and using the fact that $\beta_0=1$ for all $n$, leads to the recursion:
\begin{eqnarray*}
c_{n,n}&=&1\quad \text{for all}\quad n \quad \\
c_{n+k,n}&=& \beta_k - \sum_{i=1}^k\beta_ia_{n+k-i}^ic_{n+k-i,n}\quad \text{if}\quad 1\le k\le J \quad *\\
c_{n+k,n}&= &- \sum_{i=1}^J\beta_ia_{n+k-i}^ic_{n+k-i,n}\quad \text{if}\quad k>J \quad **\\
\end{eqnarray*}
\par
This recursion is profitably viewed as a vector recursion.  For $n \geq 0$ and $j \geq n+J$, let  
$\vec{v}_{j,n}=\left(c_{j-J+1,n}, c_{j-J+2,n}, \ldots , c_{j,n} \right)^T.$ The $J$ by $J$ matrix 
$$M_n= \begin{pmatrix} 
0                        & 1                      & 0       & \cdots & 0        &0\\
0                        & 0                      & 1       & \cdots & 0        &0\\
\vdots                &\vdots                &\ddots & \dots &\vdots &\vdots\\
0                        & 0                      & 0       & \cdots & 0        &1\\
-\beta_Ja_{n-J+1}^J &-\beta_{J-1}a_{n-J+2} ^{J-1} &     -\beta_{J-2}a_{n-J+3}^{J-2}    &\cdots       & -\beta_2a_{n-1}^2  &  -\beta_1a_{n} \\
\end{pmatrix}$$
encodes the map which takes 
$\left(c_1,c_2, \ldots , c_J\right)^T$ to 
$\left(c_2,c_3, \ldots , c_J,  -  \sum_{i=1}^J\beta_ia_{n-i+1}^ic_{J+1-i}  \right)^T.$  
This allows equation ** to be 
expressed by the recursion: $\vec{v}_{n+k,n} = M_{n+k} \vec{v}_{n+k-1,n} $ for $k >J$.   
Tracing the recursion backwards, one obtains   
$$\vec{v}_{n+k,n} = M_{n+k}  M_{n+k-1} \cdots M_{n+J+1} \, \vec{v}_{n+J,n} \quad \text{for}\quad k >J.$$
\par
The recursion matrix $M_n$ and its pointwise limit 
$$M_\infty= \begin{pmatrix} 0 & 1 & 0 & \dots & 0 & 0\\
0 & 0 & 1 & \dots & 0 & 0\\
\vdots & \vdots & \vdots & \ddots & 1 & 0\\
0 & 0 & 0  & \dots & 0 & 1\\
- \beta_J & - \beta_{J-1} & -\beta_{J-2} & \dots & -\beta_2 & -\beta_1 \end{pmatrix}$$   
will play dominant roles in what follows.  
Note that  $\vec{\nu}_j = \left(z_j^{J-1},z_j^{J-2}, \ldots ,z_j,1 \right)^T$ is an eigenvector for $M_\infty$ with eigenvalue $w_j$ 
for $j=1, \ldots, J$. It is well-known that $\{\vec{\nu}_j:j=1,2,\dots J\}$ forms a basis for $\mathbb{C}^J$, and it turns out that in the proceeding section it will be useful to describe the action of $M_n$ in terms of a basis of these eigenvectors.
\par 
To determine when $C$ is bounded, we will estimate the norms of such matrix products for large $k$. The following result due to Adams and McGuire in \cite{Adams-McGuire1} will then provide the desired condition:
\begin{theorem}[Adams-McGuire] \label{matrix theorem}
If $p>0$, then the matrix 
$$M=\renewcommand\arraystretch{1.5}\begin{pmatrix}
 0 & 0 & 0 & 0 & \dots\\
\frac{p}{2} & 0 & 0 & 0 & \dots\\
\frac{p}{2}(\frac{2}{3})^p & \frac{p}{3} & 0 & 0 & \dots\\
\frac{p}{2}(\frac{2}{4})^p & \frac{p}{3}(\frac{3}{4})^p & \frac{p}{3} & 0 & \dots\\
\frac{p}{2}(\frac{2}{5})^p & \frac{p}{3}(\frac{3}{5})^p & \frac{p}{4}(\frac{4}{5})^p & \frac{p}{5} &\dots\\
\vdots & \vdots & \vdots & \vdots & \ddots
\end{pmatrix}
$$
 is bounded  if and only if $p>\frac{1}{2}$.
\end{theorem}
The following result gives sufficient conditions on the decay of the norms of products of the matrices $M_n$ and the norms of the ``starting vectors" in order for the containment $\phi(z)H^2(\mathbb{D}) \subset H(K)$ to hold.
\begin{theorem} \label{big theorem}
If $M_n$ is the recursion matrix defined above and
for some $p>1/2$,  $\mu\in \mathbb{Z}^+$, $N \ge J$, and $D_1>0$, we have the estimate
$$|| M_{n+\mu-1} M_{n+\mu-2} \cdots M_{n} ||\leq (1-p\mu/n)$$

for all $n \ge N$, and
$$||\vec{v}_{n+J,n}||\leq D_1 \frac{p}{n+J}$$
for all $n$, then
$\phi(z)H^2(\mathbb{D}) \subset H(K)$. 
\end{theorem}
\begin{proof}
First notice that it suffices to prove that the matrix $C$ defined above is the matrix of a bounded operator on $\ell ^2$. Let $D_2=\sup_n \lVert M_n \rVert$. Note it is clear that $D_2< \infty$ as the entries in $M_n$ are uniformly bounded in $n$.

Given $n,k \in \mathbb{Z}^+$ with $k \ge N+J$, let $m$ be the largest integer such that $k-m\mu \ge N+J$.  Then $m \ge 0,$ and from the recursion
\begin{eqnarray*}
\left| c_{n+k,n} \right| & \le & \left\lVert  \vec{v}_{n+k,n} \right\rVert \\
                                  & = & \left\lVert M_{n+k}  M_{n+k-1} \cdots M_{n+k-m\mu+1} \,\vec{v}_{n+k-m\mu,n} \right\rVert \\
                                  & \leq & \left\lVert M_{n+k}  M_{n+k-1} \cdots M_{n+k-m\mu+1}  \right\rVert \lVert\vec{v}_{n+k-m\mu,n}\rVert\\
                                   & \le &  \prod_{j=1}^{m}\left( 1-p\mu/(n+k+1-j\mu) \right) \rVert  \lVert\vec{v}_{n+k-m\mu,n}\rVert\\
                                   \end{eqnarray*}
For $0<\epsilon <1$, $\log{(1-\epsilon)} < -\epsilon$. Without loss of generality we may assume $N>p \mu$, which affords
\begin{eqnarray*}
\log{ \prod_{j=1}^{m}\left( 1-p\mu/(n+k+1-j\mu) \right)} & < & \sum_{j=1}^{m}\left( -p\mu/(n+k+1-j\mu) \right) \\
                                                                            & < & \sum_{j=0}^{m-1}\left( -p\mu/(n+N+J+1+(j+1)\mu) \right) \\
& \leq & \int_{0}^{m} \left(-\frac{p\mu}{N'+\mu x}\right) dx\\
&=& -p \log(N'+\mu x) \big|_0^{m}\\
& = &\log\left(\left[\frac{N'}{N'+m \mu}\right]^{p}\right) .
\end{eqnarray*}

where $N'=n+N+J+\mu+1$.
Therefore,
\begin{eqnarray*}
|c_{n+k,n}|& \leq &  \left[\frac{N'}{N'+m \mu}\right]^{p}\lVert\vec{v}_{n+k-m\mu,n}\rVert \\
& =& \left[\frac{N'}{N'+m \mu}\right]^{p}\lVert M_{n+k-m\mu}  M_{n+k-m\mu-1,n} \cdots M_{n+J+1} \vec{v}_{n+J,n} \rVert \\
& \leq & \left[\frac{N'}{N'+m \mu}\right]^{p}D_2^{N+\mu}\lVert  \vec{v}_{n+J,n} \rVert \\
& \leq & D_2^{N+\mu} D_1\frac{p}{n+J}\left[\frac{N'}{N'+m\mu}\right]^{p}
 \end{eqnarray*}

Recalling that the Schur or Hadamard product of a bounded matrix with another matrix with entries bounded away from $0$ and $\infty$ is bounded (see Lemma 2.1 in \cite{Adams-McGuire1}), a simple application of the preceding theorem demonstrates that $C$ is bounded.

\end{proof}
\par

\section[Finite Bandwidth Reproducing Kernels]{Finite Bandwidth Reproducing Kernels}
In this section, we obtain an explicit decomposition for these spaces in analogy with \cite{Adams-McGuire1} in the case $p>1/2$ and $\lim_{n \rightarrow \infty} n(1-a_n)=p$. In doing so we substantially extend their results to arbitrary bandwidths and more general weight sequences.

\par
The following two lemmas have routine proofs and are needed for the purposes of computation.
\begin{lemma}\label{entrywise product bounds} If $A_1,A_2,\dots,A_k$ are $n \times n$ matrices with complex entries bounded  in modulus by $c$ then 
 $$||A_1 \dots A_k|| \leq n^k c^k$$
  \end{lemma}
\begin{lemma}\label{ergodic lemma} If $z_1, z_2, \ldots , z_J$ are points on the unit circle $\mathbb{T}$,
then $(1,1, \ldots ,1) \in \mathbb{C}^J$ is a limit point of the set  $\{\left( z_1^\mu,  z_2^\mu, \ldots , z_J^\mu \right): \mu \in \mathbb{Z}^{+}\}$.
\end{lemma} 
\begin{proof}
Repeatedly apply the compactness of $\mathbb{T}$.
\end{proof}
\par
We now proceed to the statement and proof of the main lemma. 
\begin{lemma}\label{big lemma} Let $M_n$ denote the recursion matrix defined above, $\{a_n\}$ a non-vanishing sequence satisfying $\lim_{n \rightarrow \infty} n(1-a_n)=p$ where $p>1/2$, , and  $X$ the change of basis matrix whose $j$th column is the eigenvector 
$\vec{\nu}_j$ of the limiting matrix $M_\infty. $ If $\widehat{M}_n=X^{-1}M_n X,$
then for all $\varepsilon>0,$ there exist positive integers $\mu$ and $N$ such that for all $n>N$ $$||\widehat{M}_{n+\mu-1}\ldots \widehat{M}_{n}|| \leq 1-\frac{(\mu p-\varepsilon)}{n}.$$
\end{lemma}
 \begin{proof}
 \par 
 
Let $\mu$ be a large positive integer to be chosen later and fix $k$   with $0 \leq k < \mu-1.$ We will choose $N$ later based on an appropriate choice of $\mu$. Linearize $M_{n+k}$ by writing $M_{n+k}=M_\infty+(p/n)B+R_{n,k},$ where
$B$ is the $J$ by $J$ matrix whose first $J-1$ rows are zero and whose last row is
$$ \begin{pmatrix} 
J \beta_J & (J-1)\beta_{J-1} & (J-2)\beta_{J-2} &\dots & 2\beta_2 & \beta_1 \end{pmatrix}$$ 
and $R_{n,k}$ is the $J$ by $J$ matrix whose first $J-1$ rows are zero and whose $J$th row is
$$\begin{pmatrix}
\left(1-a_{n-J+k+1}^J-\frac{pJ}{n}\right) \beta_J & \dots & \left(1-a_{n-1+k}^2-\frac{2p}{n}\right)\beta_2 & \left(1-a_{n+k}-\frac{p}{n}\right)\beta_1 \end{pmatrix}$$
\par
Since $R_{n,k}$ can be bounded entrywise by $\frac{E(n)}{n} $, where $E(n)$ is some function satisfying $\lim_{n\rightarrow \infty} E(n) =0,$ it follows by Lemma \ref{entrywise product bounds} that $||R_{n,k}||\le \frac{J E(n)}{n}$. We compute
\begin{eqnarray*}
\widehat{M}_{n+\mu-1}\ldots \widehat{M}_{n}&=& X^{-1}  \prod_{k=0}^{\mu-1} (M_\infty+\frac{pB}{n}+R_{n,k})X \\
&=& X^{-1}\left( M_\infty^{\mu}+ \sum_{k=0}^{\mu-1} M_\infty^k \frac{pB}{n} M_\infty^{\mu-1-k} + R \right) X,
\end{eqnarray*}
\noindent where $R$ is the sum of all products in the expansion involving the matrices $R_{n,k}$. (There are $3^{\mu}-\mu-1$ such terms). Thus,  $||X^{-1}RX|| < \frac{C_1 E(n)}{n}$ where $C_1$ is a constant that depends only on $J$ and $\mu$.
\par
The crucial norm estimate will come from 
$$ X^{-1}\left(M_\infty^{\mu}+\sum_{k=0}^{\mu-1} M_\infty^k \frac{B}{n} M_\infty^{\mu-1-k}\right)X,$$ so we turn to a computation of this norm.
A straightforward Gaussian elimination shows that the vector $\vec{\nu}_0 = \left( 0,0, \ldots , 0,1\right)$ can be expressed in terms of the eigenvectors for $M_\infty$ as $\sum_{j=1}^J \, - w_j/\phi^\prime (z_j)\vec{\nu}_j$.
\par
To compute the norm of $X^{-1}\left(M_\infty^{\mu} + \sum_{k=0}^{\mu-1} M_\infty^k \frac{B}{n} M_\infty^{\mu-1-k}\right)X$, consider the action of 
$\sum_{k=0}^{\mu-1} M_\infty^k \frac{B}{n} M_\infty^{\mu-1-k}$ on $\vec{\nu}_h$ for $h \in \{1,2, \ldots , J\}$.  
Note that
\newline
$\phi(z)=1+\sum_{k=1}^J \beta_kz^k = \prod_{j=1}^J (1-w_j z)$ and notice that
$$\phi^\prime(z_h)=-w_h \prod_{j:j \ne h} (1-w_j z_h)=\sum_{k=1}^J k\beta_kz_h^{k-1} .$$
Now, $z_j$ is on the unit circle, so $ (1-w_j z_h) =w_j (z_j-z_h).$
\newline
Thus,  

$$\phi^\prime(z_h) =(-\prod_{j=1}^J w_j   ) \prod_{j:j \ne h} (z_j - z_h) .$$

Therefore,  
\begin{eqnarray*}
B\vec{\nu}_h & = &  \phi^\prime(z_h) \vec{\nu}_0\\
&=& \phi^\prime(z_h) \sum_{j=1}^J  - w_j/\phi^\prime (z_j)\vec{\nu}_j\\
&=&-w_h \vec{\nu}_h- \sum_{j:j\ne h}  w_j \frac{ \phi^\prime(z_h)}{\phi^\prime (z_j)}\vec{\nu}_j\\
\end{eqnarray*}
Thus,
\begin{eqnarray*}
\sum_{k=0}^{\mu-1} M_\infty^k \frac{pB}{n} M_\infty^{\mu-1-k} \vec{\nu}_h &=& \sum_{k=0}^{\mu-1} w_h^{\mu-1-k}M_\infty^k \frac{pB}{n} \vec{\nu}_h\\
 &=&- \frac{p}{n}w_h^{\mu-1}\sum_{k=0}^{\mu-1} w_h^{-k}M_\infty^k 
 \left(w_h \vec{\nu}_h + \sum_{j:j\ne h}  w_j \frac{ \phi^\prime(z_h)}{\phi^\prime (z_j)}\vec{\nu}_j\right)\\
 &=&- \frac{p}{n}w_h^{\mu-1}\sum_{k=0}^{\mu-1} w_h^{-k} 
 \left(w_h^{k+1} \vec{\nu}_h + \sum_{j:j\ne h}  w_j^{k+1} \frac{ \phi^\prime(z_h)}{\phi^\prime (z_j)}\vec{\nu}_j\right)\\
 &=&- \frac{\mu p}{n}w_h^{\mu}\vec{\nu}_h
+ \sum_{j:j\ne h}  - \frac{p}{n}\frac{w_j}{w_h^{1-\mu}}\left( \frac{1-(w_j/w_h)^\mu}{1-w_j/w_h}\right) \frac{ \phi^\prime(z_h)}{\phi^\prime (z_j)}\vec{\nu}_j\\
 \end{eqnarray*}
 By Lemma \ref{ergodic lemma}, for each $\varepsilon >0$, there is a $\mu \in \mathbb{N}$ such that each of the modulus of each of coefficients of
 $v_j$ for $j \neq h$ above is less than  $\frac{\varepsilon}{2Jn}$.

Since
$M_\infty^\mu \vec{\nu}_h=w_h^\mu \vec{v_h}$,
it follows that the norm of $X^{-1}\left(M_\infty^{\mu}+\sum_{k=0}^{\mu-1} M_\infty^k \frac{B}{n} M_\infty^{\mu-1-k}\right)X$ is bounded above by the norm of the matrix
$$P=\begin{pmatrix} \left(1-\frac{\mu p}{n}\right) & \frac{\varepsilon}{2Jn} & \frac{\varepsilon}{2Jn} & \frac{\varepsilon}{2Jn}& \dots & \frac{\varepsilon}{2Jn}\\
\frac{\varepsilon}{2Jn} & \left(1-\frac{\mu p}{n}\right) & \frac{\varepsilon}{2Jn}  &\frac{\varepsilon}{2Jn}  &\dots & \frac{\varepsilon}{2Jn}  \\
\frac{\varepsilon}{2Jn}  & \frac{\varepsilon}{2Jn}  & \left(1-\frac{\mu p}{n}\right)& \frac{\varepsilon}{2Jn}  &\dots & \frac{\varepsilon}{2Jn}  \\
\frac{\varepsilon}{2Jn}  & \frac{\varepsilon}{2Jn}  & \frac{\varepsilon}{2Jn}   & \left(1-\frac{\mu p}{n}\right)& \dots & \frac{\varepsilon}{2Jn} \\
\vdots & \vdots & \vdots & \vdots & \ddots & \vdots\\
\frac{\varepsilon}{2Jn}  & \frac{\varepsilon}{2Jn}  & \frac{\varepsilon}{2Jn}   &  \frac{\varepsilon}{2Jn}  & \cdots &  \left(1-\frac{\mu p}{n}\right)\\
 \end{pmatrix}$$

But from the triangle inequality we have the estimate $$||P|| \leq  \left(1-\frac{\mu p}{n}\right) + \frac{\varepsilon}{2n}$$
\newline
\color{black}
 Putting all of our calculations together and choosing $N$ large enough so that for $n>N$, $E(n)<\frac{\varepsilon}{2C_1}$, we deduce that, for all $n>N:$
$$\left\Vert\widehat{M}_{n+\mu-1}\ldots \widehat{M}_{n}\right\Vert \leq  1-\frac{\mu p}{n} + \frac{\varepsilon}{2n} + \frac{\varepsilon}{2n}= 1-\frac{(\mu p-\varepsilon)}{n} \ .$$

\end{proof}

Now we are ready to prove the containment result.

\begin{theorem} \label{containment}
 If $H(K)$ denotes the reproducing kernel Hilbert space with orthonormal basis $$f_{n}(z)=\phi(a_nz)z^n$$ satisfying $p > 1/2$ and $\lim_{n\rightarrow \infty} n(1-a_n)=p$, then    $\phi(z)H^2(\mathbb{D}) \subset H(K)$.
\end{theorem}

\begin{proof}
This is a simple application of Theorem \ref{big theorem} and Lemma \ref{big lemma}. First, choose $\varepsilon>0$ sufficiently small so that $p-\varepsilon>1/2$.  By Lemma \ref{big lemma}, there exist positive integers $\mu$ and $N$ such that for all $n>N$ $$||\widehat{M}_{n+\mu-1}\ldots \widehat{M}_{n}|| \leq 1-\frac{(\mu p-\varepsilon)}{n}=1-\frac{\mu p'}{n},$$
where $p'=p-\frac{\varepsilon}{\mu}>1/2$.
Note 
\begin{eqnarray*}
\left \lVert M_{n+\mu-1} M_{n+\mu-2} \cdots M_{n}  \right\rVert 
                                  & = & \left\lVert X \widehat{M}_{n+\mu-1} \widehat{M}_{n+\mu-2}  \cdots \widehat{M}_{n}X^{-1}\ \right\rVert \\
                                  & \leq & \left\lVert \widehat{M}_{n+k}  \widehat{M}_{n+k-1} \cdots \widehat{M}_{n+k-m\mu+1}  \right\rVert 
                                   \lVert X \rVert   \lVert X^{-1} \rVert \\
                                   & \le & 
\lVert X \rVert   \lVert X^{-1} \rVert                                   \left( 1-\frac{\mu p'}{n} \right) \\
                                   \end{eqnarray*}
The extra constant is harmless in regards to the proof of Theorem \ref{big theorem}.\\
                                    
It only remains to check the growth rate on the starting vectors $\vec{v}_{n+J,n}$, using our previous notation. We claim that for each $1 \leq j \leq J$, there exists a bounded sequence of complex numbers $\{\alpha_{n,j}\}_n$, such that for all $n \in \mathbb{N}$, $c_{n+j,n}=(1-a_n)\alpha_{n,j}$. Note that this implies there exists a positive real constant $M$
such that $||\vec{v}_{n+J,n}||\leq M|1-a_n|$, which in turn implies the starting vectors satisfy the growth rate of Theorem \ref{big theorem}.

We prove the claim by induction on $j$. For the base case, note that $c_{n+1,n}=\beta_1-a_n \beta_1 c_{n,n}=\beta_1(1-a_n)$. Then notice that \begin{eqnarray*}
c_{n+j,n}& = & \beta_j(1-a_n^j)-\sum_{i=1}^{j-1}\beta_{i}a_{n+j-i}^i c_{n+j-i,n}\\
& = & \beta_j(1+a_n+a_n^2+\dots+a_n^{j-1})(1-a_n)-\sum_{i=1}^{j-1}\beta_{i}a_{n+j-i}^i(1-a_n)\alpha_{n,j-i}\end{eqnarray*}

By induction, the claim holds.

As the hypotheses of Theorem \ref{big theorem} are evidently satisfied, the containment follows. 
\end{proof}
\begin{example} \label{bad example}
This example shows that if $a_n \rightarrow 1$ more rapidly then $a_n =1-p/n$, then the containment of the previous result does not occur.  Specifically, if $J=2$, $z_1=1$,  $z_2=-1$, and  $a_n=1-\frac{1}{(n+2)^2}$, then
 $(1-z)(1+z) H^2(\mathbb{D}) \subseteq H(K)$ if and only if there is a bounded matrix $C$ satisfying $\hat{L}=LC$, where 

$$ \begin{pmatrix} 
1 & 0 & 0 & \cdots \\
0 & 1 & 0& \cdots\\
-1 & 0 & 1 & \cdots \\
0 & -1 & 0 & \ddots\\
0 & 0 & -1 & \ddots\\
\vdots & \vdots & \vdots &  \ddots
 \end{pmatrix} 
 =
  \begin{pmatrix} 
1 & 0 & 0 &\cdots \\
0 & 1 & 0& \cdots\\
-\frac{9}{16}& 0 & 1 & \cdots\\
0 & -\frac{64}{81}& 0 & \ddots\\
0 & 0 & -\frac{225}{256} & \ddots\\
\vdots & \vdots & \vdots &  \ddots
 \end{pmatrix}  \begin{pmatrix} 
c_{0,0} &0 &0 &\cdots \\
c_{1,0} & c_{1,1} & 0 & \cdots\\
c_{2,0} & c_{2,1} & c_{2,2} & \ddots\\
c_{3,0} & c_{3,1} & c_{3,2} & \ddots\\
c_{4,0} & c_{4,1} & c_{4,2} & \ddots\\
c_{5,0} & c_{5,1} & c_{5,2} &  \ddots\\
\vdots & \vdots & \vdots & \ddots\\

 \end{pmatrix} $$
 
The entries of  $C$ are completely determined by this equation and it is straightforward to show that  $\lim c_{n,0}\neq 0$ and thus that $C$ is not bounded. The same argument works for $a_n=1-\frac{1}{(n+2)^p}$ with $p>1$.
\end{example}
\par
Before tackling the second half of the decomposition, a few different results will be required. First, to ensure this decomposition actually makes sense we need to establish that the natural domain of $H(K)$, which we denote by $\mathcal{D}$, of $H(K)$ consists of the unit disc $\mathbb{D}$ plus the $J$ ``extra" points on the boundary $z_1,z_2,\dots,z_J$. 
\begin{proposition}\label{natural domain} If $\mathcal{D}$ denotes the natural domain of the space $H(K)$, then $$\mathcal{D}=\mathbb{D} \cup \{z_1,z_2, \ldots z_J \}$$

\begin{proof}
It suffices to verify that for $1 \leq j \leq J$ we have $\sum_{n=0}^{\infty}|f_n(z_j)|^2<\infty$. But this is clear, as $\sum_{n=0}^{\infty}|f_n(z_j)|^2 \lesssim \sum_{n=0}^{\infty} |1-a_n|^2$ which is comparable to  $\sum_{n=0}^{\infty}\frac{p^2}{n^2}<\infty$.

\end{proof}

\end{proposition}

Next, we proceed to state two technical propositions that we will need in the forthcoming proof. The proofs are postponed to the next section.  The second theorem relies on results from the theory of symmetrical polynomials.

\begin{proposition}\label{independence}
The matrix $A$ defined by
$$A=\begin{pmatrix}K(z_1,z_1) & K(z_2,z_1) &\cdots & K(z_J,z_1)\\ 
                                 K(z_1,z_2) & K(z_2,z_2) &\cdots & K(z_J,z_2)\\
                                 \vdots &   \vdots& \cdots &   \vdots\\
                                 K(z_1,z_J) & K(z_2,z_J) &\cdots & K(z_J,z_J)
         \end{pmatrix}$$
is invertible.
\end{proposition}
\begin{proposition} \label{QRecursion}
For $ j \in \{1,2, \ldots J\}$, define 
$$\mu_j = \prod_{ k \neq j} (w_j-w_k) $$
If
$$Q_n (x) = \sum_{j=1}^J \, \frac{w_j^J}{\mu_j} \phi(x /w_j)w_j^n,$$
then  $Q_0(x), Q_1(x), \ldots $ satisfy the recursion:
$$
 \sum_{i=0}^n \beta_i Q_{n-i}(x)   =  \beta_{n+1}( x^{n+1}-1)$$
\end{proposition}

\begin{theorem}\label{decomposition} For every $f \in H(K),$ there exists a $g \in H^2(\mathbb{D})$
and constants $b_1,b_2, \ldots, b_J \in \mathbb{C}, $ such that 
$$f(z)=\phi(z)g(z)+ b_1 K(z,z_1)+\cdots+ b_J K(z,z_J)  .$$
\end{theorem}
\begin{proof}   Given $f \in H(K)$, first choose $b_1,b_2, \ldots , b_J$ so that 
 $$f(z)-b_1 K(z,z_1)-b_2 K(z,z_2)- \cdots -b_J K(z,z_J) $$ 
 vanishes at $z=z_1, \ldots , z_J.$ Note this is always possible in light of Proposition \ref{independence}.
Thus,
assume, without loss of generality,  
that $f \in H(K)$ satisfies $f(z_1)=f(z_2) = \cdots =f(z_J)=0$ for $j= 1,2, \ldots , J$. Our goal now becomes to demonstrate the existence of a $g \in H^2(\mathbb{D})$ so $f=\phi g$.
\newline
As $f \in H(K),$  
there exists $\{\alpha_n \} \in \ell^{2}$  such that
$$f(z)= \sum_{n=0}^{\infty}\alpha_n f_n (z) .$$
We shall refer to such a sequence $\{ \alpha_n \}$ as permissable.
We will produce a sequence $\{g_n\} \in \ell^{2}$ such that 
$$f(z)=\phi(z)\left(\sum_{n=0}^{\infty}g_n z^n\right).$$
Expanding both expressions for $f$ and equating gives:
$$  \sum_{n=0}^{\infty} \sum_{k=0}^{J} \alpha_n a_n^k \beta_k z^k z^n =
      \sum_{n=0}^{\infty} \sum_{k=0}^{J} g_n \beta_k z^k z^n$$
Equating like powers of $z$ above leads to the 
equation
$$  \sum_{k=0}^J \alpha_{n-k}\beta_k a_{n-k}^k -g_{n-k}\beta_k = 0 \hspace{.1in} \text{for}\hspace{.1in} n=0,1,2,\ldots. $$
where any quantities with negative subscripts are treated as zero.  
Since $\beta_0=1$, this relationship can be expressed as the recursion:
$$\text{*\hspace{.2in}      }g_n=\alpha_n+\left(\sum_{j=n-J}^{n-1}\alpha_j \beta_{n-j} a_j^{n-j}  - g_j\beta_{n-j} \right).$$  

Recursion * shows that one may express $g_j$ as a linear combination, 
$$g_n= \sum_{k=0}^{n} c_{n,k} \alpha_{k},$$ 
for some constants $c_{n,k}.$ 
\par
Applying *   and equating like coefficients leads to 
$$ c_{n,n} =1, $$
$$ c_{n,k} = \beta_{n-k} a_k^{n-k} -\sum_{i=1}^{n-k} \beta_i c_{n-i,k}  \hspace{0.75 cm} n-J \leq k \leq n-1 ,$$

and for $0 \leq k \leq n-J-1,$
$$c_{n,k}= - \sum_{i=1}^{J} \beta_i c_{n-i,k}.$$
This suggests that one
let $\{ p_n : n\in \mathbb{Z}_+\}$ be the sequence of polynomials defined by the linear recursion:
$$ p_0(x) =1, $$
$$ p_1(x) =-\beta_1(1-x), $$
$$\vdots$$
$$p_n(x)= \beta_n x^n-\sum_{i=1}^{n} \beta_i p_{n-i}(x)$$
$$\vdots$$
$$p_J(x)=\beta_Jx^J-\sum_{i=1}^{J}\beta_ip_{J-i}(x)$$
and thereafter, if $n \geq J+1,$
$$\text{**\hspace{.2in}      }p_n(x)=- \sum_{i=1}^{J} \beta_i p_{n-i}(x).$$
Then
$$c_{n+k,k}=p_{n}(a_{k}) \text{\hspace{ .2in } if }n\geq 0.$$
To prove this claim, notice that it follows directly for all $k \ge 0$ if $n=0,1,\dots, J$ using induction.  The cases $n >J$ then follow from the recursion by induction.
\par
Thus the map $\{\alpha_n\} \mapsto \{g_n\}$ is encoded by the following matrix $B_p$ (that is, $\{g_n\}_{n=0}^{\infty}=B_p \{\alpha_n\}_{n=0}^{\infty}$) where
$$B_p=\begin{pmatrix}
1&0& 0 & 0 & 0 & 0& \dots\\
p_1(a_0)&1& 0 & 0 & 0 & 0& \dots\\
p_2(a_0) &p_1(a_1) & 1&0  & 0 & 0  &\ddots\\
p_3(a_0) &p_2(a_1) & p_1(a_2)& 1&0 & 0 & \ddots\\
p_4(a_0) & p_3(a_1) & p_2(a_2) & p_1(a_3) & 1&0 &  \ddots\\
\vdots & \ddots & \ddots& \ddots & \ddots & \ddots & \ddots
\end{pmatrix}$$
\par
If the matrix $B_p$ were bounded as an operator, then the desired result would follow immediately.  
However, the columns of $B_p$ are not in $\ell^2$.  We will use the assumption that $f(z_j)=0$ for $j=1,2, \dots, J$, to find an 
equivalent encoding of the map $\{\alpha_n\} \mapsto \{g_n\}$ which is bounded.  
\par
To find this alternate encoding of $B_p$, begin by considering the vector  
$$\vec{v}_n=\begin{pmatrix}
p_n(a_0) & p_{n-1}(a_1) & \cdots&p_2(a_{n-2}) & p_1(a_{n-1}) & 1&0 &  \cdots\\
\end{pmatrix}$$
which equals the n'th row of $B_p$. 
Let $z_j$ be a root of $\phi$.  The fact that $f(z_j)=0$ is equivalent to
the equation $\sum_{n=0}^{\infty}\alpha_n \phi(a_n z_j)z_j^n=0$ which in turn means that the 
vector
$$\vec{w}_j=\begin{pmatrix}
\phi(a_0 z_j) & \phi(a_1 z_j) z_j& \phi(a_2 z_j)z_j^2 & \phi(a_3 z_j)z_j^3   &  \cdots\\
\end{pmatrix} \quad \text{for }  j \in \{ 1,2, \ldots J\}.$$
is orthogonal to any permissible $\vec{\alpha}=(\alpha_n)_{n=0}^\infty.$
\par
Let $q_{j,n}(x)=\phi(x z_j)z_j^{-n}$  for $n \in \mathbb{Z}_+$.
Then the polynomial sequence $\{ q_{j,n} :n\in \mathbb{Z} \}$  satisfies condition ** satisfied by 
$\{ p_n :n\in \mathbb{Z}_+ \}$. (This follows directly from the fact that $z_j$ is a root of $\phi$.)
Moreover, the vector 
$$\vec{u}_j=\begin{pmatrix}
q_{j,n}(a_0 ) &q_{j,n-1}(a_1 ) &&\ldots&q_{j,1}(a_{n-1} )  &q_{j,0}(a_n ) & q_{j,-1}(a_{n+1} )  &  \cdots\\
\end{pmatrix} $$
equals $w_j^n \vec{w}_j$ and thus is orthogonal to all permissible sequences.
\par
Therefore, the $n$th row $\vec{v}_n$ of $B_p$ can be replaced by $\vec{v}_n$ less any linear combination of 
the vectors $\vec{u}_1, \vec{u}_2, \ldots \vec{u}_J$ without changing the action on permissible vectors.
Proposition \ref{QRecursion} shows that subtracting $\vec{v'}_n =\left( Q_{n-1}(a_0), Q_{n-2}(a_1)), Q_{n-3}(a_2), \ldots \right)$ from $\vec{v}_n$
zeroes out the first $n$ entries.  Thus, an equivalent encoding of $B_p$ is given by the matrix
$${C}=\begin{pmatrix}
1-Q_{-1}(a_0) &-Q_{-2}(a_1) & -Q_{-3}(a_2)&-Q_{-4}(a_3)&  \dots\\
0 & 1-Q_{-1}(a_1)  &  -Q_{-2}(a_2) &  -Q_{-3}(a_3) \dots\\
0 & 0 &  1-Q_{-1}(a_2)  & -Q_{-2}(a_3) \ddots\\
0 &0 & 0  & 1-Q_{-1}(a_3)   &  \ddots\\
\cdots & \cdots & \cdots & \ddots & \ddots 
\end{pmatrix}.$$

Since $w_1, w_2, \ldots , w_J$ are discrete points on the unit circle, it is a straightforward exercise to show that there exists a constant $c$, independent of $m$ and $n$,  such that
$|Q_n(a_m)| \le c(1-a_m) .$
\par
Thus the map $\{\alpha_j\} \mapsto \{g_j\}$ is bounded if the matrix $\widehat{C}$ is bounded where
$$\widehat{C}=\begin{pmatrix}
1-a_0 & 1-a_1 &  1-a_2 &  \dots\\
0 & 1-a_1 & 1-a_2  & \ddots\\
0 &0 & 1-a_3 &  \ddots\\
\cdots & \cdots & \ddots & \ddots 
\end{pmatrix}.$$

But this matrix is known to be bounded since the entries behave asymptotically like $\frac{p}{n}$ (see Theorem 2.2 in \cite{Adams-McGuire1}), establishing the result.
\end{proof}

\begin{remark} Note that the preceding result is independent of $p$ (it holds for all $p>0$). Compare this to Theorem \ref{containment}.
\end{remark}

\begin{remark} Note that the proof of the preceding theorem demonstrates that if we had taken $a_js$ with a slower convergence rate, we would not have obtained a bounded matrix for $\hat{C}$. In particular, suppose that $a_j=1-\left(\frac{1}{j+2}\right)^p$ where $p<1/2 $. Then we would obtain 

$$\widehat{C}=\begin{pmatrix}
\frac{1}{2^p} & \frac{1}{3^p} &  \frac{1}{4^p}&  \dots\\
0 & \frac{1}{3^p} & \frac{1}{4^p}  & \ddots\\
0 &0 & \frac{1}{4^p} &  \ddots\\
\cdots & \cdots & \ddots & \ddots 
\end{pmatrix}.$$

This matrix is easily seen to be unbounded ( in particular the $\ell^2$ norms of its columns approach $\infty$), which suggests (but does not prove) that we might not obtain the result of the theorem in this case. Together with Example \ref{bad example}, this helps justify the consideration of spaces with the specific growth rate given in the hypothesis of the theorem. 
\end{remark}

Theorem \ref{decomposition} admits the following corollary, completing our characterization of these spaces when $p>\frac{1}{2}$ and $\lim_{n \rightarrow \infty} n(1-a_n)=p$:

\begin{corollary}
If $p>1/2$ and $\lim_{n \rightarrow \infty} n(1-a_n)=p$, then $$H(K)= \phi(z) H^2(\mathbb{D})+\mathbb{C} K(z,{z}_1)+\mathbb{C} K(z,{z}_2) + \cdots  +\mathbb{C} K(z,{z}_J).$$
\end{corollary}
\section[Proof of combinatorial propositions]{Proof of combinatorial propositions}

\begin{lemma}
If $f_n(z)=\phi(a_n z)z_j^n$ is the $n$th basis vector for $H(K)$,
then for some $n$, 
the matrix 
$$B_n=
\begin{pmatrix} 
 f_n(z_1) &  f_n(z_2) & \cdots &  f_n(z_J) \\
 f_{n+1}(z_1) &   f_{n+1}(z_2) & \cdots &   f_{n+1}(z_J) \\
\vdots&\vdots&\vdots&\vdots\\
 f_{n+J-1}(z_1) &   f_{n+J-1}(z_2) & \cdots &   f_{n+J-1}(z_J) \\
\end{pmatrix}$$
is invertible.
\end{lemma}
\begin{proof}
Define $\phi_j(z) =\prod_{k \ne j} (1-w_k z)$ and 
notice that $f_n(z_j)=\phi_j(a_n z_j)z_j^n(1-a_n ).$
Notice that $B_n$ can be written as the product  $B_n=D_1 C_n D_2$ where $D_1$ is the diagonal matrix with entries 
$1-a_n,1-a_{n+1}, \ldots 1-a_{n+J-1}$ and $D_2$ is the diagonal matrix with entries $z_1^{n+1}, z_2^{n+1}, \ldots z_J^{n+1}$.
Thus, 
$$C_n=\left( \phi_j(a_{n+i}z_j)z_j^{i-1}  \right) _{i,j=1}^J$$
Notice that the component-wise limit of $C_n$ as $n \rightarrow \infty$
is 
$$C_\infty=\left( \phi_j(z_j)z_j^{i-1}  \right) _{i,j=1}^J,$$
which is the matrix product of the Vandermonde matrix  
$V=\left(z_j^{i-1}  \right) _{i,j=1}^J$ with the diagonal matrix
$D_3$ with entries $ \phi_1(z_1),  \phi_2(z_2), \ldots ,  \phi_J(z_J)$.
Since these matrices are invertible, so too is $C_\infty$.  
Since the invertible matrices form an open set set in $\mathbb{C}^{J^2}$,
$C_n$ must be invertible for some $n.$
\end{proof}
\begin{proof}[\bf Proof of Proposition \ref{independence}]
Suppose that $ A\vec{v} =\vec{0}$ for some $\vec{v} \in \mathbb{C}^J$.  Then
$$0 = \langle A \vec{v} , \vec{v} \rangle =|| \sum_{k=1}^J v_k K(z, z_k) ||^2$$
But, this implies that $\sum_{k=1}^J v_k K(z, z_k) =0$.
\par
Use the preceding lemma to find $J$ elements $g_1, g_2, \ldots, g_J$ of $H(K)$ with the property that
$g_j(z_k)= 0$, if  $k\ne j$ and $g_j(z_j)=1$.
Thus, 
$$\bar{v}_j =\sum_{k=1}^J \langle g_j(z) ,  v_k K(z, z_k) \rangle= \langle g_j(z) , \sum_{k=1}^J v_k K(z, z_k) \rangle =  \langle g_j(z) ,0 \rangle =0.$$
In other words, $A$ has trivial kernel, so must be invertible.
\end{proof}

The following two theorems from combinatorics provide the necessary tools to prove Proposition \ref{QRecursion}. Theorem \ref{Louck} appears in \cite{Louck} while Theorem \ref{homoSum} is a well-known result in combinatorics.
\begin{theorem}\label{Louck}  [See \cite{Louck} Theorem 2.2.]
For each integer $m \ge 0$,
$$\sum_{j=1}^J\, x_j^m/\mu_j  = h_{m-J+1}(x_1,x_2, \ldots, x_J),$$
where $h_k$ is the $k$'th homogeneous symmetric polynomial, which is defined to be zero for $k<0$.
\end{theorem}
\begin{theorem}\label{homoSum} 
For each integer $m > 0$,
$$\sum_{i=0}^m\, \beta_i  h_{m-i}(x_1,x_2, \ldots, x_J)=0.$$

\end{theorem}

Theorem \ref{homoSum} is a well-known result in the field of symmetric polynomials and we omit its proof. Now we are in a position to prove Proposition \ref{QRecursion}:

\begin{proof}[\bf Proof of Proposition \ref{QRecursion}]
First assume $0 \le n <J$, and write
$$ \sum_{i=0}^n \beta_i Q_{n-i}(x) = \sum_{k=0}^J a_k x^k.$$
Then
\begin{eqnarray*}
 \sum_{i=0}^n \beta_i Q_{n-i}(x)  & = & \sum_{i=0}^n \beta_i \sum_{j=1}^J \, \frac{w_j^J}{\mu_j} \phi(x /w_j)w_j^{n-i}\\
 &=& \sum_{i=0}^n \beta_i \sum_{j=1}^J \, \sum_{k=0}^J \frac{w_j^J}{\mu_j} \beta_k\left( \frac{x}{w_j}\right)^k w_j^{n-i}\\
  &=& \sum_{k=0}^J  \beta_k x^k \sum_{i=0}^n \beta_i \sum_{j=1}^J \, \frac{w_j^{J+n-i-k}
  }{\mu_j}  \\
    &=& \sum_{k=0}^J  \beta_k x^k \sum_{i=0}^n \beta_i h_{n-k-i+1}(w_1, \ldots , w_J)\\
 \\
\end{eqnarray*} 
Thus, 
\begin{eqnarray*}
a_0 & = & \beta_0  \sum_{i=0}^n \beta_i h_{n-i+1}(w_1, \ldots , w_J).\\
\end{eqnarray*} 
Now $\beta_0=1$ and from Theorem 2, $ \sum_{i=0}^{n+1} \beta_i h_{n-i+1}(w_1, \ldots , w_J)=0$.  Thus, $a_0=-\beta_{n+1}.$
\par
Now suppose $1\le k \le n$.  Then
\begin{eqnarray*}
a_k & = & \beta_k  \sum_{i=0}^n \beta_i h_{n-k-i+1}(w_1, \ldots , w_J)\\
&= & \beta_k  \sum_{i=0}^{n-k+1} \beta_i h_{n-k-i+1}(w_1, \ldots , w_J)\\
&= &0.\\
\end{eqnarray*} 
For $k=n+1$,
\begin{eqnarray*}
a_{n+1} & = & \beta_{n+1}  \sum_{i=0}^n \beta_i h_{-i}(w_1, \ldots , w_J)=\beta_{n+1}\\
\end{eqnarray*} 
since only the first term in the sum is non-zero.
\par
If $n+1<k<J$,  then $n-k-i+1$ is always negative for $i\ge 0$ so 
\begin{eqnarray*}
a_{k} & = & \beta_{k}  \sum_{i=0}^n \beta_i h_{n+1-k-i}(w_1, \ldots , w_J)=0.\\
\end{eqnarray*} 
This shows that recursion * holds for $0\le n < J$.  
\newline
Now, suppose $n \ge J.$
Then,
\begin{eqnarray*}
 \sum_{i=0}^n \beta_i Q_{n-i}(x)  &=& \sum_{k=0}^J  \beta_k x^k \sum_{i=0}^n \beta_i h_{n-k-i+1}(x_1, \ldots , x_J)\\
 \\
\end{eqnarray*} 
Since $n \ge J$, and $\beta_j = 0$ for $j >J$, Theorem 2 applies to show that the sum
$ \sum_{i=0}^n \beta_i Q_{n-i}(x) $ equals zero.
\end{proof}

\section{Some Additional Consequences}

Consider next the natural question of whether $H(K)$ is closed under multiplication by the independent variable $z$. We have the following result:

\begin{theorem}\label{zMultiplies}
If $p > \frac{1}{2}$ and $\lim_{n \rightarrow \infty}n(1-a_n)=p$, then $z$ is a multiplier on $H(K).$
 \end{theorem} 

\begin{proof} It is sufficient to show that the matrix representation of $M_z$ with respect to the orthonormal basis 
$\{ f_n:n\in\mathbb{N}\}$ is bounded as a matrix. Denote this matrix as $C=(c_{k,n})$. Thus 
$$M_z(f_n)=\sum_{k=0}^{\infty}c_{k,n}f_k$$ with the coefficients $c_{k,n}$ yet to be determined.  Expanding the sum and rearranging as powers of $z$ shows that $c_{k,n}=0$ for $k\leq n$ and leads to the recursion:
\begin{eqnarray*}
c_{n+1,n} & = & 1\\
c_{n+j+1,n} & = & \beta_j a_n^j-\sum_{i=1}^{j} \beta_{i} a_{n+j+1-i}^{i}c_{n+j+1-i,n} \quad \text{if}\quad 0 \leq j \leq J\\
c_{n+J+k+1,n} & = & -\sum_{i=1}^{J} \beta_i a_{n+J+k+1-i}^i c_{n+J+k+1-i,n} \quad \text{if}\quad 1 \leq k \\
\end{eqnarray*}
Notice that for $k \ge 1$, this is precisely the same recursion encoded by $M_n$ and  Theorem \ref{containment}
applies to demonstrate the boundedness of $C$ (as before, it is straightforward to show the starting vectors have the appropriate decay and we omit the details, just note that the diagonal of $1$s can be removed without affecting the boundedness of $C$).
\end{proof}

Thus, in addition to establishing that the multiplier algebra of $H(K)$ contains the polynomials, we get the following nice result:

\begin{corollary}
 Let $H(K)$ denote the reproducing kernel Hilbert space with orthonormal basis $$f_{n}(z)=\phi(a_nz)z^n.$$   If $p>1/2$ and $\lim_{n \rightarrow \infty} n(1-a_n)=p$, then $H(K)$ contains the polynomials.
\end{corollary}

\begin{proof}
In light of Theorem \ref{zMultiplies}, it suffices to show that $1\in H(K)$.
Write 
$$1=\sum_{n=0}^\infty \, c_nf_n(z) =\sum_{n=0}^\infty \,\left( \sum_{j=0}^J \, c_n\beta_j a_n^j  z^{j+n}\right).$$
It is enough to show $(c_n) \in \ell^2$. Equating like powers of $z$ leads to the recursion with starting value $c_0=1$  and thereafter:
$$c_j = -  \sum_{i=1}^ j c_{j-i} \beta_i a_{j-i}^i \quad \text{if} \quad j \geq 1$$
where we recall that $\beta_i =0$ if $i > J$.  Once again, the vectors 
$\vec{v}_{n}=\left(c_{n-J+1}, c_{n-J+2}, \ldots , c_{n} \right)^T$  satisfy the recursion
$\vec{v}_{n+1} =M_{n+1}\vec{v}_n$ for $n=J,J+1,\ldots$ and the result follows as before.
\end{proof}

Much future work could be done in this area. For instance, one could try to obtain a full characterization of the multiplier algebras of these finite bandwidth spaces.

\section*{Acknowledgments} 
We would like to acknowledge Paul McGuire for his careful reading of the manuscript and numerous suggestions. We would also like to acknowledge Cody Stockdale, whose honors thesis at Bucknell University in 2015 paved the way for the key matricial methods used in this paper. This paper grew out of the second author's Honors Thesis at Bucknell University in 2017. The second author is currently supported by a NSF GRF (grant number DGE-1745038).

\end{document}